\documentclass[11pt]{amsart}
\usepackage{amsmath, amsthm}
\usepackage{amssymb}
\usepackage{amscd}
\usepackage{url}
\usepackage{graphicx}
\usepackage{color}

\numberwithin{equation}{section}

\theoremstyle{plain}

\newtheorem{theorem}[subsection]{Theorem}
\newtheorem{proposition}[subsection]{Proposition}
\newtheorem{lemma}[subsection]{Lemma}
\newtheorem{corollary}[subsection]{Corollary}

\theoremstyle{definition}
\newtheorem{definition}[subsection]{Definition}

\newcommand{\Q}{\mathbb{Q}}
\newcommand{\Z}{\mathbb Z}
\newcommand{\C}{\mathbb C}

\newcommand{\R}{\mathbb{R}}

\title{The H\"older exponent of some Fourier series}
\author{Fernando Chamizo}
\author{Izabela Petrykiewicz}
\author{Seraf\'{\i}n Ruiz-Cabello}

\address[F. Chamizo]{Department of Mathematics, Universidad Aut\'{o}noma de Madrid and ICMAT, 28049 Madrid, Spain}
\email{fernando.chamizo@uam.es}
\thanks{The first and the third authors are partially supported by the grant MTM2011-22851 of the Ministerio de Ciencia e Innovaci\'{o}n (Spain).}

\address[I. Petrykiewicz]{Max Planck Institute for Mathematics, Vivatsgasse 7, 53111 Bonn, Germany}
\email{petrykii@mpim-bonn.mpg.de}

\address[S. Ruiz-Cabello]{Department of Mathematics, Universidad Aut\'{o}noma de Ma\-drid, 28049 Madrid, Spain}
\email{serafin.ruiz@uam.es}



\begin{document}

\begin{abstract}
In this paper we study the local regularity of fractional integrals of Fourier series using several definitions of the H\"older exponent. We especially consider series coming from fractional integrals of modular forms. Our results show that in general cusp forms give rise to pure fractals (as opposed to multifractals). We include explicit examples and computer plots.
\end{abstract}

\maketitle


\section{Introduction}

There are many ways of classifying a continuous function according to its regularity giving rise to several definitions of function spaces. Some of them can be adapted to study the regularity at a single point. 
For instance, the classical Lipschitz spaces $\Lambda^s$ (see \cite{zygmund}) lead to define for $0<s\le 1$
\begin{equation}\label{lipschitz}
\Lambda^s(x_0)=
\big\{f\text{ continuous function}\;:\; 
\big|f(x)-f(x_0)\big|=O\big(|x-x_0|^s\big)\big\}.
\end{equation}

The following extension to any $s>0$ is important in multifractal analysis:

\begin{definition}
Given $x_0\in\R$, we define $C^{s}(x_0)$, with $s\ge 0$, as the set of continuous functions
$f:\R\longrightarrow\C$ such that there exists a polynomial $P$ of degree at most $s$ satisfying
\begin{equation}\label{cs}
 \big|f(x)-P(x-x_0)\big|=O\big(|x-x_0|^s\big)
\qquad\text{when }x\to x_0.
\end{equation}
For a fixed $f$ we define the \emph{pointwise H\"older exponent} of $f$ at $x_0$ as
\begin{equation}\label{phe}
 \beta(x_0)=\sup\big\{s\ge 0\;:\; f\in C^{s}(x_0) \big\}.
\end{equation}
\end{definition}

Clearly $C^s(x_0)=\Lambda^s(x_0)$ for $0<s\le 1$. 
One is tempted to consider $P$ as the Taylor polynomial but, as we shall see later, $f\in C^s(x_0)$ for $s$ large does not even assure the existence of $f''(x_0)$.
 
In analysis, a more common extension of the Lipschitz spaces $\Lambda^s$ are the H\"older spaces $C^{k,s}$ where $k$ is a nonnegative integer and $0<s\le 1$. 
This is the space of continuous functions having continuous derivatives up to order $k$ and such that $f^{(k)}\in \Lambda^s$. This leads to a more naive way of defining the regularity at a point.

\begin{definition}
Let $C^{k,s}(x_0)$ be the set of continuous functions $f:\R\longrightarrow\C$ such that $f^{(k)}$ is continuous in an open interval $I$ containing $x_0$ and such that $f^{(k)}\in\Lambda^s(x_0)$.
Given a continuous function $f$ we define the \emph{restricted local H\"older exponent} of $f$ at $x_0$ as
\begin{equation}\label{rlhe}
 \beta^*(x_0)=\sup\big\{k+s\;:\; f\in C^{k,s}(x_0),\ k\ge 0,\ 0<s\le 1 \big\}.
\end{equation}
\end{definition}

In some sense, the definition of $\beta^*$ is like imposing in \eqref{cs} that $P$ is the actual $k$-th order Taylor polynomial.

Finally, if the localization is thought in the topological or analytic sense as a kind of limit of open neighborhoods, then it is better to avoid \eqref{lipschitz} and to employ $\Lambda^s$ as it appears in the definition of $C^{k,s}$
but restricted to an open set. In connection with this, for an open set $\mathcal{U}\subset\R$, we denote as usual by $C^{k,s}(\mathcal{U})$ the set of $k$-differentiable functions such that 
$\big|f^{(k)}(x)-f^{(k)}(y)\big|=O\big(|x-y|^s\big)$ for every $x,y\in\mathcal{U}$.
The definition extends to $s=0$ meaning the set of functions with $k$ continuous derivatives and no extra conditions.

\begin{definition}
Let $I_1\supset I_2\supset\, I_3\supset\ldots$ be a sequence of open nested intervals in $\R$ such that $\bigcap I_n = \{x_0\}$. 
Given a continuous function $f$ we define the \emph{local H\"older exponent} of $f$ at $x_0$ as
\begin{equation}\label{lhe}
 \beta^{**}(x_0)=
\lim_{n\to\infty}
\sup\big\{k+s\;:\; f\in C^{k,s}(I_n),\ k\ge 0,\ 0\le s\le 1 \big\}.
\end{equation}
\end{definition}

It is not difficult to see that $\beta^{**}(x_0)$ does not depend on the choice of the nested intervals \cite{SeVe}.

\

The difference between these definitions becomes apparent when considering chirp functions. For instance, take $f(x)=x^4\sin(x^{-2})$ for $x\ne 0$ and $f(0)=0$.
We can choose $P$ identically zero and $s=4$ in \eqref{cs}, and no greater values of $s$, then $\beta(0)=4$. The derivative, $f'(x)=4x^3\sin(x^{-2})-2x\cos(x^{-2})$ for $x\ne 0$ and $f'(0)=0$, is clearly nondifferentiable at $x_0=0$. 
On the other hand $f'\in\Lambda^1(0)$, in particular $f\in C^{1,1}(0)$ and $\beta^*(0)=2$.
The definition of $\beta^{**}$ is more demanding because $f'$ has to belong locally to a $\Lambda^s$. In our case choosing $x_n=\big((n+1)\pi\big)^{-1/2}$ and $y_n=(n\pi)^{-1/2}$, 
we have $\big|f'(x_n)-f'(y_n)\big|>C n^{-1/2}$ for some $C>0$. Hence $\big|f'(x_n)-f'(y_n)\big|\ne o\big(|x_n-y_n|^{1/3}\big)$ and $f'\not\in \Lambda^s$ for $s>1/3$. 
In fact it is not difficult to see that this is the limiting case and~$\beta^{**}(0)=4/3$.

\

Comparing the definitions, the following result is straightforward:
\begin{lemma}\label{compare}
For a continuous function $f:\R\longrightarrow\C$ consider the H\"older exponents defined in \eqref{phe}, \eqref{rlhe}, \eqref{lhe}. Then for any $x_0\in\R$
\[
 \beta(x_0)\ge 
 \beta^{*}(x_0)\ge 
 \beta^{**}(x_0).
\]
\end{lemma}

The previous example shows that we cannot expect equalities in general. In fact the pointwise H\"older exponent and the local H\"older exponent can be~$\infty$ and $0$, respectively, at the same point, for instance this is the case at $x_0=0$ for the function $f(x)=e^{-{x^{-2}}}\sin\big( e^{{x^{-4}}}\big)$, with $f(0)=0$.

\

As mentioned in \cite{SeVe}, the local H\"older exponent is arguably the most natural from the analytic point of view because of its stability under the action of pseudo-differential operators (having in mind especially fractional derivatives).

\

In this paper, we consider functions that are fractional integrals of Fourier series. These are classical objects in harmonic analysis coming at least from \cite{HaLi2}. In Section~\ref{section:genres}, we provide conditions assuring that the pointwise H\"older exponent, the restricted local H\"older exponent and the local H\"older exponent of these functions coincide (see \cite{SeVe} for general relations between the first and the third).
The conditions are met by certain series arising from modular forms.
These and other related series has been treated by several authors in connection with fractal and multifractal analysis
\cite{chamizo}, \cite{HoTc}, \cite{jaffard}, \cite{MiSc}, \cite{otal}, \cite{petrykiewicz1}, \cite{ruizcabello}.
We deal with them in Section~\ref{section:modular}
and we deduce that the spectra of singularities of fractional integrals of cusp forms have a discrete image, in this sense they are pure fractals. It is interesting to note that it follows from \cite{petrykiewicz1} and \cite{ruizcabello} that, in contrast, in some ranges the fractional integrals of modular forms which are not cusp forms are multifractals.
The pointwise H\"older exponent of series arising from modular forms has been studied by the second author in \cite{petrykiewicz1}, however only the case of irrational points was considered. At the end of Section~\ref{section:modular}, we provide the complementary results by computing the pointwise H\"older exponent at rational points.
Finally, in Section~4 we illustrate the results with some examples and computer plots. 

\section{General results}\label{section:genres}

In this section we state some results related to fractional integrals of Fourier series. The general situation is as follows: we consider a sequence of complex numbers indexed by integers 
$\{a_n\}_{n\in\Z}$ with at most polynomial growth and we introduce the fractional integral of the Fourier series with Fourier coefficients given by this sequence. Namely, we use the notation
\begin{equation}\label{frac_int}
f_\alpha(x)=\sum_{n\ne 0}\frac{a_n}{|n|^\alpha}e(nx),
\qquad\text{where \quad $e(x)=e^{2\pi i x}$,}
\end{equation}
and we assume the absolute convergence, in this way $f_\alpha$ is a continuous function. Note that the polynomial growth of $a_n$ implies that this is always the case for large enough $\alpha$. 

In principle the Fourier series $\sum a_n e(nx)$ is not well-defined but we can always consider its Poisson integral $\sum a_n e(nx)r^{|n|}$ with $0\le r<1$ which, under mild convergence conditions on the Fourier series, admits an integral representation in terms of the Poisson kernel \cite{rudin}
\begin{equation}\label{po_id}
\sum_{n=-\infty}^\infty a_n e(nx)r^{|n|}
=
\int_{-1/2}^{1/2}
P_r(x-t)
\sum_{n=-\infty}^\infty a_n e(nt)\;dt
\end{equation}
with
\begin{equation}\label{fo_po_ke}
P_r(t)=
\Re
\Big(
\frac{1+re(t)}{1-re(t)}
\Big)
=
\frac{1-r^2}{1-2r\cos(2\pi t)+r^2}.
\end{equation}
Equivalently, renaming $r=e^{-2\pi y}$, we have the extension of the Fourier series to the upper half plane
\begin{equation}\label{poi_int}
f(z)
=
\sum_{n=-\infty}^\infty a_n e(nx)e^{-2\pi |n|y}
=
\sum_{n=0}^\infty a_n e(nz)
+
\sum_{n=-\infty}^{-1} a_n e(n\overline{z})
\end{equation}
where $z=x+iy\in\mathbb{H}$.

With our assumption on the growth of $a_n$, the convergence of these series is assured.

\

The following result shows the coincidence of the different definitions of the H\"older exponent under certain conditions that appear naturally in examples coming from automorphic cusp forms:

\begin{theorem}\label{th:general}
Assume that, for a fixed sequence $\{a_n\}_{n\in\Z}$ and $\gamma> 0$,
\begin{equation}\label{ebound}
\sum_{0\le n\le N} a_ne(nx)=O_{\varepsilon}\big(N^{\gamma+\varepsilon}\big)
\quad\text{and}\quad
\sum_{-N\le n\le 0} a_ne(nx)=O_{\varepsilon}\big(N^{\gamma+\varepsilon}\big)
\end{equation}
hold for any $\varepsilon>0$, uniformly in $x\in\R$. If for a given $x_0\in\R$ we have $f(x_0+iy)\ne o\big(y^{-\gamma}\big)$ with $f$ as in \eqref{poi_int}, then 
\[
\beta(x_0)=\beta^{*}(x_0)=\beta^{**}(x_0)=\alpha-\gamma
\]
where $\beta(x_0)$, $\beta^{*}(x_0)$, $\beta^{**}(x_0)$ are the H\"older exponents of the function $f_\alpha$ defined by \eqref{frac_int} with $\alpha>\gamma$. 
\end{theorem}

\begin{proof}
Take $m=\alpha-\gamma-1$ if $\alpha-\gamma$ is an integer and $m=\lfloor \alpha-\gamma\rfloor$ otherwise. By the triangle inequality,
\begin{eqnarray*}
 \frac{\big|f_\alpha^{(m)}(x)-f_\alpha^{(m)}(y)\big|}{(2\pi)^m}
 &\le&
 \sum_{\delta\in\{-1,1\}}\Big|
 \sum_{0<\delta n\le N}
 \frac{a_n}{|n|^{\alpha-m}}
 \big(e(nx)-e(ny)\big)
 \Big|
 \\
 &&
 +
 \Big|
 \sum_{|n|>N}
 \frac{a_n}{|n|^{\alpha-m}}
 e(nx)
 \Big|
 +
 \Big|
 \sum_{|n|>N}
 \frac{a_n}{|n|^{\alpha-m}}
 e(ny)
 \Big|.
\end{eqnarray*}
The two latter sums are $O\big(N^{\gamma-\alpha+m+\varepsilon}\big)$ by \eqref{ebound} after partial summation. By the mean value theorem applied to the real and imaginary parts of the first inner sum, and using again \eqref{ebound}, we have for $\delta = \pm 1$
\begin{eqnarray*}
 \Big|
 \sum_{0<\delta n\le N}
 \frac{a_n}{|n|^{\alpha-m}}
 \big(e(nx)-e(ny)\big)
 \Big|
 &\le& 4\pi|x-y|
 \Big|
 \sum_{0<|n|\le N}
 \frac{a_n}{|n|^{\alpha-m-1}}
 e(n\theta)
 \Big|
 \\
 &=&O\big(|x-y|N^{m-(\alpha-\gamma)+1+\varepsilon}\big).
\end{eqnarray*}
Choosing $N$ like $|x-y|^{-1}$ we deduce that $f_\alpha\in C^{m,\alpha-\gamma-m-\varepsilon}(I)$ for any bounded interval $I\subset\R$ and, by Lemma~\ref{compare}, $\beta(x_0)\ge  \beta^{*}(x_0)\ge \beta^{**}(x_0)\ge \alpha-\gamma$.

It remains to prove that $\beta(x_0)>\alpha-\gamma$ leads to a contradiction. In that case there exists $\varepsilon>0$, that we assume less than $\gamma$, such that 
\begin{equation}\label{new_bound_rev}
 \left| f_\alpha(x) - Q(x-x_0) \right|  
 = O\left(|x-x_0|^{\alpha-\gamma+\varepsilon}\right)
\end{equation}
for a certain polynomial $Q$ of degree at most $\alpha-\gamma+\varepsilon$.
Let us fix an even integer $m'$ greater than $\alpha+\max\{0,-\gamma+\varepsilon\}$.
We are going to prove (cf. \cite[Lemma~2.11]{hardy}, \cite[Lemma~3.6]{chamizo})
\begin{equation}\label{condition}
\sum_{n\neq 0} \frac{a_n}{|n|^{\alpha-m'}} e(nx_0)e^{-2\pi|n|y}  = O(y^{\alpha-\gamma+\varepsilon-m'})
\qquad \text{as }\quad y \to 0^+.
\end{equation}

By \eqref{po_id}, the left hand side is, up to a constant, 
$\int_{-1/2}^{1/2}P_r^{(m')}(t)f_\alpha(x_0-t)\; dt$ where $r=e^{-2\pi y}$, $r\to 1^-$.
By \eqref{new_bound_rev}, $f_\alpha(x_0-t)= 
Q(-t)+O\left(|t|^{\alpha-\gamma+\varepsilon}\right)$
and the contribution of this latter $O$-term to the integral is admissible using the bounds \cite[(3.9)]{chamizo}
\[
 P_r^{(k)}(t)=
 \begin{cases}
  O\big((1-r)|t|^{-k-2}\big)&\text{for }1-r<|t|\le 1/2
\\  O\big((1-r)^{-k-1}\big)&\text{for }|t|\le 1-r
 \end{cases},
 \qquad k\in\Z_{\ge 0}.
\]
On the other hand, as $Q(-t)$ is a polynomial of degree less than $m'$, by repeated integration by parts, its contribution to the integral is   given by the sum $\sum_{k=0}^{m'-1} Q^{(k)}(-t)P_r^{(m'-1-k)}(t)\big|_{-1/2}^{1/2}=O(1)$ and the proof of \eqref{condition} is complete.

By direct integration  term by term and using \eqref{condition}, we obtain
\begin{align*}
(f(x_0+iy)-a_0)y^{\alpha-m'}  & = C\int_0^{\infty}t^{m'-\alpha-1}
\sum_{n\neq 0} \frac{a_n}{|n|^{\alpha-m'}}e(nx_0)e^{-2\pi |n|y(1+t)}\,dt 
\\
  & = \int_0^{\infty} t^{m'-\alpha-1}O\left([(1+t)y]^{\alpha-\gamma+\varepsilon-m'}\right)\,dt,
\end{align*}
The last integral is convergent since $m'-\alpha$ and $\gamma-\varepsilon$ are positive. Multiplying by $y^{m'-\alpha}$ we get $f(x_0+iy) = O(y^{-\gamma+\varepsilon})$
that contradicts the assumption in the statement of the theorem.
\end{proof}

\

If $a_n=0$ for $n<0$, the function $f$ defined in \eqref{poi_int} becomes holomorphic in the upper half plane. In this case we define
\[
f_\alpha^c(x)=\sum_{n=1}^\infty \frac{a_n}{n^\alpha}\cos(2\pi nx)
\qquad\text{and}\qquad
f_\alpha^s(x)=\sum_{n=1}^\infty \frac{a_n}{n^\alpha}\sin (2\pi nx).
\]
The regularity of $f_\alpha^c$ and $f_\alpha^s$ is linked to the functions (again holomorphic in the upper half plane)
\[
f^{\Re}(z)
=
\sum_{n=0}^\infty \Re(a_n) e(nz)
\qquad\text{and}\qquad
f^{\Im}(z)
=
\sum_{n=0}^\infty \Im(a_n) e(nz).
\]

The following result is a version of Proposition~1 of \cite{jaffard2} employing the analytic wavelet from \cite{petrykiewicz1}:

\begin{proposition}\label{pr:jaffard}
Assume $a_n=0$ for $n<0$, as before, and let $\beta>0$ with $\lfloor\beta\rfloor\leq \alpha-1$. If any of the functions $f_\alpha$, $f_\alpha^c$ and $f_\alpha^s$ belongs to $C^\beta(x_0)$, then
\begin{equation}\label{wbound1}
f(x+iy)=
O\big( y^{\beta-\alpha}\big(1+y^{-1}|x-x_0|\big)^\beta\big)
\qquad\text{when }(x,y)\to (x_0,0^+).
\end{equation}
On the other hand, if for some $0<\beta'<\beta$ we have
\begin{equation}\label{wbound2} 
f(x+iy)=
O\big( y^{\beta-\alpha}\big(1+y^{-1}|x-x_0|\big)^{\beta'}\big)
\qquad\text{when }(x,y)\to (x_0,0^+),
\end{equation}
then $f_\alpha\in C^\beta(x_0)$. Moreover, if $f^{\Re}$ and $f^{\Im}$ both satisfy \eqref{wbound2}, then also $f_\alpha^s,f_\alpha^c\in C^\beta(x_0)$.
\end{proposition}

\begin{proof}
As in \cite{petrykiewicz1}, we consider $\psi_\alpha(x)=(x+i)^{-\alpha-1}$ that verifies the properties required in \cite{jaffard2} to be an analytic wavelet. Namely
\begin{align*}
 &
 |\psi_\alpha(x)|+
 |\psi'_\alpha(x)|+
 \dots+ |\psi^{(\lfloor \beta\rfloor+1)}_\alpha(x)|
 = O\big(|x|^{-\lfloor \beta\rfloor-2}\big)
 \quad\text{as }x\to\infty,
 \\
 &
 \int_{-\infty}^\infty \psi_\alpha(x)\; dx=
 \int_{-\infty}^\infty x\psi_\alpha(x)\; dx=
 \dots =  \int_{-\infty}^\infty x^{\lfloor \beta\rfloor}\psi_\alpha(x)\; dx=0,
 \\
 &
 \widehat{\psi}_\alpha(\xi)=0\text{ for $\xi\le 0$ }
 \quad\text{and}\quad
 \int_0^\infty \xi^{-1}\big|\widehat{\psi}_\alpha(\xi)\big|^2\; d\xi<\infty.
\end{align*}
These properties are checked in \cite{petrykiewicz1}.
The last integral is assumed to be~$1$ in the original definition in \cite{jaffard2}, but it is harmless multiplying $\psi_\alpha$ by a scaling constant. 

The corresponding wavelet transform is 
\begin{equation}\label{wa_tr}
 Wg(a,b)
 =a^\alpha\int_{-\infty}^\infty
 g(t)(t-b-ia)^{-\alpha-1}\; dt\quad \text{ with $a>0$ and $b\in\R$.}
\end{equation}
It is not other that the classical Riemann-Liouville fractional integral \cite{HaLi1}.

By the residue theorem, the wavelet transform of $g(t)=e(\lambda t)$ is $0$ if $\lambda\le 0$ and $C_\alpha \lambda^\alpha a^\alpha e\big(\lambda(b+ia)\big)$ otherwise, with $C_\alpha$ a constant depending on $\alpha$. By the Euler formula, the wavelet transforms of $\cos(2\pi\lambda t)$ and $\sin(2\pi\lambda t)$ are also of the form 
$C_\alpha \lambda^\alpha a^\alpha e\big(\lambda(b+ia)\big)$ for any $\lambda>0$. Hence, for $g=f_\alpha,f_\alpha^c,f_\alpha^s$, \eqref{wa_tr} reads
\[
 Wg(y,x)
 = C_\alpha y^\alpha \big(f(z)-a_0\big)
\quad \text{ where $z=x+iy$.}
\]

The function $f_\alpha$ satisfies $\widehat{f}_\alpha(\xi)=0$ for $\xi<0$, in the distributional sense. Then Proposition~1 of \cite{jaffard2} gives the results concerning $f_\alpha$, namely, if $f_\alpha\in C^\beta(x_0)$, then $f$ satisfies \eqref{wbound1} and conversely if $f$ satisfies \eqref{wbound2}, then $f_\alpha\in C^\beta(x_0)$ (see also the Remark in \cite[p.161]{HoTc}).
This proposition also applies when one considers real valued functions and we can proceed in a similar way considering 
$\Re(f_\alpha^c)$ and $\Re(f_\alpha^s)$, with wavelet transform 
$C_\alpha y^\alpha \big(f^{\Re}(z)-\Re(a_0)\big)$, or 
$\Im(f_\alpha^c)$ and $\Im(f_\alpha^s)$, with wavelet transform 
$C_\alpha y^\alpha \big(f^{\Im}(z)-\Im(a_0)\big)$. 

If $f_\alpha^c\in C^\beta(x_0)$, then $\Re \big(f_\alpha^c\big), \Im \big(f_\alpha^c\big)\in  C^\beta(x_0)$, hence by \cite{jaffard2}, $f^{\Re}$ and $f^{\Im}$ satisfy the bound \eqref{wbound1} and consequently $f=f^{\Re}+f^{\Im}$ too. On the other hand, if $f^{\Re}$ and $f^{\Im}$ satisfy \eqref{wbound2}, then \cite{jaffard2} gives $\Re \big(f_\alpha^c\big), \Im \big(f_\alpha^c\big)\in  C^\beta(x_0)$ that implies $f_\alpha^c\in  C^\beta(x_0)$. 
The same arguments hold for $f_\alpha^s$.
\end{proof}

\section{Fractional integrals of modular forms}\label{section:modular}

With a view to generalizing the Riemann's example (see below) we consider functions coming from the Fourier expansion of modular forms.

We assume that $\Gamma$ is a congruence group, i.e., $\Gamma(N)<\Gamma<\text{SL}_2(\mathbb{Z})$ where $\Gamma(N)$ is the principal congruence subgroup. 
Given a real number $r>0$, we say that $f$ is a modular form of weight $r$ with respect to $\Gamma$ if the following equality holds for every $\gamma\in\Gamma$:
\begin{equation}\label{eq:modrel}
 f(\gamma z) = m_{\gamma}\big(j_{\gamma}(z)\big)^rf(z).
\end{equation}
Here, $m_{\gamma}$ is a multiplier system verifying $|m_{\gamma}|=1$ and $j_{\gamma}(z)$ 
is the denominator of the linear fractional transformation corresponding to $\gamma$.
The function $f$ is required to be holomorphic not only in the upper complex half plane, but also at the cusps of $\Gamma$ (see \cite[\S2]{iwaniec}). 

In our statements in this section, for the sake of simplicity, we also assume that the width of the cusp at $i\infty$ is~1, this implies that the integral translations belong to $\Gamma$.   
Under these conditions any modular form with respect to such these groups admits a Fourier expansion,
\[
 f(z) = \sum_{n\geq 0}a_ne(nz).
\]
Formally other widths are covered changing $z$ by $z/q$.

\begin{theorem}\label{pr:modforms}
If $f(z)=\sum_{n=1}^\infty a_ne(nz)$ is a cusp form of weight $r>0$, for any $\alpha>r/2$ the function
$f_\alpha$
verifies $\beta(x_0)=\beta^{*}(x_0)=\beta^{**}(x_0)=\alpha-r/2$ for $x_0$ irrational. If the coefficients $a_n$ are real then the same holds for  $f_\alpha^c$ and $f_\alpha^s$.
\end{theorem}

\begin{proof}
By \cite[Lemma~3.2]{chamizo} we have
\[
 \sum_{n=0}^{N} a_ne(nx)=O_\varepsilon\big(N^{r/2+\varepsilon}\big)
\qquad\text{for all $\varepsilon>0$}.
\]
Since $x_0$ is not a cusp, \cite[Lemma~3.4]{chamizo} implies that
$f(x_0+iy)\ne o(y^{-r/2})$
and the result for $f_\alpha$ is a consequence of Theorem~\ref{th:general} choosing $\gamma=r/2$.

If the coefficients are real, clearly
$f_\alpha \in C^{k,s}(\mathcal{U})$ implies that $f_\alpha^c, f_\alpha^s\in C^{k,s}(\mathcal{U})$. Therefore, it suffices to show that $\beta(x_0) \leq \alpha -r/2$ for both functions. 
Assume that there exists $0<\varepsilon\leq {r}/{2}-1$ such that $f_\alpha^c \in C^{\alpha - r/2 +\varepsilon}$. 
By Proposition~\ref{pr:jaffard}, we have $f(x+iy)=O(y^{-r/2+\varepsilon}(1+y^{-1}|x-x_0|)^{\alpha-r/2+\varepsilon})$ as $(x,y) \to (x_0,0^+)$. 
In particular 
$f(x_0+iy)= O(y^{-r/2+\epsilon})$ that contradicts
$f(x_0+iy)\ne o(y^{-r/2})$. 
The same argument holds for $f_\alpha^s$, which completes the proof of the theorem.
\end{proof}

\

Let $f$ be a continuous complex valued function. The spectrum of singularities of $f$ is defined as the following correspondence:
\[
   d_f(\delta) := \begin{cases} \dim_{\mathcal{H}}\{x\in[0,1]:\beta(x)=\delta\}, & \text{if $\{x\in[0,1]:\beta(x)=\delta\}\neq \emptyset$}. \\ -\infty, & \text{if $\{x\in[0,1]:\beta(x)=\delta\}= \emptyset$},  \end{cases}
\]
where $\dim_{\mathcal{H}}$ is the Hausdorff dimension and $\beta(x)$ is the pointwise H\"older exponent of $f$. If the image of $d_f$ is not discrete, we say that $f$ is a multifractal function, or just a multifractal.

\

\begin{corollary}\label{cr:modforms}
With the notation and assumptions of Theorem~\ref{pr:modforms}, the 
functions $f_\alpha$, $f_\alpha^c$ and $f_\alpha^s$ 
are not multifractal.
\end{corollary}

\begin{proof}
For $x_0\in\R\setminus\Q$ we have $\beta(x_0)=\alpha-r/2$ by Theorem~\ref{pr:modforms}. As $\Q$ is a set of zero Hausdorff dimension, the spectrum of singularities can take at most 3 values.
\end{proof}

\

Once we have determined $\beta(x_0)$ at irrational values the natural question is what happens when $x_0\in\Q$.

\begin{theorem}\label{pr:cusp_rat}
If $f(z)=\sum_{n=1}^\infty a_ne(nz)$ is a cusp form of weight $r>0$, $\alpha>r/2$, then $f_\alpha$
verifies $\beta^{**}(x_0)=\alpha-r/2$ for $x_0\in\Q$. 
If, in addition, $\alpha\ge 1+\lfloor2\alpha-r\rfloor$, then $\beta(x_0)=2\alpha-r$ for every $x_0\in\Q$. 
Moreover, if the coefficients~$a_n$ are real, then the same holds for $f_\alpha^c$ and $f_\alpha^s$.
\end{theorem}

\begin{proof}
Repeating the manipulations of the first part of proof of Theorem~\ref{th:general}, we have $\beta^{**}(x_0)\geq \alpha-r/2$. On the other hand, by \cite[Proposition~4.1]{SeVe} 
we have that $\beta^{**}(x)\leq\min(\beta(x),\liminf_{t\to x}\beta(t))$ 
(noting that even though Proposition~4.1 is stated for real-valued functions the proof follows only from the definitions, hence it remains valid for complex-valued functions).
We can then choose a sequence of irrational numbers $\{t_n\}_n$ converging to $x_0$. 
By Theorem~\ref{pr:modforms}, $\beta(t_n)= \alpha-r/2$, and it follows that $\beta^{**}(x_0)\leq \alpha-r/2$.
The case of $f_\alpha^c$ and $f_\alpha^s$ does not add anything new.

We now consider the pointwise H\"older exponent. 
Assume $\beta(x_0)>2\alpha-r$, then $f_\alpha\in C^{2\alpha-r+\delta}(x_0)$.
As the integral part is upper continuous, taking $\delta$ small enough, $\lfloor 2\alpha-r+\delta\rfloor\le \alpha-1$. 
By Proposition~\ref{pr:jaffard} with $x=x_0+\eta y^{1/2}$ where $\eta$ is a small positive constant,
\[
 f\big(x_0+\eta y^{1/2}+iy\big)
 =
 O\Big( y^{\alpha-r+\delta}
 \big(1+y^{-1/2}\big)^{2\alpha-r+\delta}\Big)
 =O\big( y^{(\delta-r)/2}\big).
\]
On the other hand, by the Fourier expansion at the cusp $x_0$ (see \cite[Lemma~3.3]{chamizo}), we have that 
$\big|f\big(x_0+\eta y^{1/2}+iy\big)\big|\big|\eta y^{1/2}\big|^r\ne o(1)$ which contradicts the previous bound. 

We conclude that $\beta(x_0)\le 2\alpha-r$ and it remains to prove that the strict inequality $\beta(x_0)<2\alpha-r$ leads to a contradiction.

If $\beta(x_0)=2\alpha-r-2\delta$ for $f_\alpha$,  $f_\alpha^c$ or $f_\alpha^s$ with $0<\delta\le \alpha-r/2$, then $f_\alpha\not\in C^{2\alpha-r-\delta}(x_0)$ and by Proposition~\ref{pr:jaffard}
\[
f(x+iy) \ne
O\big( y^{\alpha-r-\delta}\big(1+y^{-1}|x-x_0|\big)^{2\alpha-r-\delta}\big)
\quad\text{when }(x,y)\to (x_0,0^+).
\]
This means that there exist two sequences 
$x_n\to x_0$ and $y_n\downarrow 0$, such that
\begin{equation}\label{limsub}
 \lim \big|f(x_n+iy_n)\big| 
 y_n^{r+\delta-\alpha}\big(1+y_n^{-1}|x_n-x_0|\big)^{r+\delta-2\alpha}
 =\infty.
\end{equation}
Choosing a subsequence, we can always assume that either $y_n^{-1}|x_n-x_0|$ is bounded or it tends to~$\infty$. In the first case $y_n^{-1}|x_n-x_0|^2\to 0$ and by \cite[Lemma~3.3]{chamizo}, $\big|f(x_n+iy_n)\big|=O\big(e^{-K/y_n}\big)$ for a certain $K>0$, which contradicts \eqref{limsub}. 
In the second case, \eqref{limsub} implies $\big|f(x_n+iy_n)\big| y_n^{r+\delta-\alpha}\to\infty$. In our ranges, this gives $\big|f(x_n+iy_n)\big| y_n^{-r/2}\to\infty$ that contradicts \cite[(5.3)]{iwaniec}.

\end{proof}

The restricted H\"older exponent can be determined under the same conditions. Note that in this case there is an unexpected dependance on the fractional part of $\alpha-r/2$, denoted by $\{\alpha-r/2\}$.

\begin{theorem}\label{pr:cusp_restricted}
With the notation and hypothesis of Theorem~\ref{pr:cusp_rat}, including $\alpha\ge 1+\lfloor2\alpha-r\rfloor$, 
for every $x_0\in\Q$ the  restricted H\"older exponent of $f_\alpha$ is 
\[
 \beta^*(x_0) =
   \lfloor 2\alpha-r\rfloor +\min\big( 1, 2\{\alpha-r/2\}\big).
\]
\end{theorem}

The proof of this result requires especial considerations
when $\alpha-r/2$ is a positive integer. In this case we need the following lemma that, in some sense, completes \cite{chamizo} (cf. Corollary~2.1.1) including an extremal case. We have not found a simple proof. The one given at the end of this section involves two other auxiliary results.

\begin{lemma}\label{notc1}
If $f$ is a cusp form of weight $r\ge4$, then $f_{r/2+1}\not\in C^{1,0}(I)$ for any open interval $I\subset\R$.
\end{lemma}

\begin{proof}[Proof of Theorem~\ref{pr:cusp_restricted}]
By Theorem \ref{pr:modforms} and Theorem \ref{pr:cusp_rat} we deduce that $f_\alpha$ is
$\lfloor \alpha -r/2 \rfloor$ times differentiable on $\R$ and, since $\alpha'=\alpha-\lfloor\alpha -r/2\rfloor >r/2$, for $\alpha-r/2\not\in\Z^+$, we have
\begin{equation}\label{eq:intder}
f_\alpha^{(\lfloor\alpha-r/2\rfloor)}
=
(2\pi i)^{\lfloor\alpha-r/2\rfloor}
f_{\alpha'}.
\end{equation}
It follows from our assumptions that $\alpha'\geq1+\lfloor2\alpha'-r\rfloor$. Hence by Theorem~\ref{pr:cusp_rat} we have
$\beta(x_0)=2\alpha'-r=2\{\alpha-r/2\}$ for $f_{\alpha'}$ and $x_0\in\Q$. 

If $0\ne \{\alpha-r/2\}\leq1/2$, then $\beta(x_0)\leq1$ implies that $\beta^*(x_0)=\beta(x_0)=2\{\alpha-r/2\}$  for $f_{\alpha'}$ and, by \eqref{eq:intder}, $\beta^*(x_0)=\lfloor \alpha-r/2\rfloor+2\{\alpha-r/2\}$ for $f_\alpha$ and $x_0\in\Q$.

If $\{\alpha-r/2\}> 1/2$,
then $\beta(x_0)>1$ for $f_{\alpha'}$, in particular we have $f_{\alpha'}\in \Lambda^1(x_0)$ and $\beta^*(x_0)\ge 1$ for $f_{\alpha'}$, using \eqref{rlhe} with $k=0$. 
If $\beta^*(x_0)> 1$ for $f_{\alpha'}$, then $f_{\alpha'}$ is differentiable in an open set, contradicting Corollary~2.1.1 of \cite{chamizo}. Hence $\beta^*(x_0)= 1$ for $f_{\alpha'}$ 
and we conclude $\beta^*(x_0)= 1+\lfloor\alpha-r/2\rfloor$ for $f_\alpha$
from \eqref{eq:intder}.

Finally,  if $\{\alpha-r/2\}=0$, we use
$f_\alpha^{(\alpha-r/2-1)}=(2\pi i)^{\alpha-r/2-1}f_{r/2+1}$ instead of \eqref{eq:intder}
and the same argument as in the last paragraph works if we prove that $f_{r/2+1}$ is not continuously differentiable on any open set.  The condition $\alpha\ge 1+\lfloor2\alpha-r\rfloor$ reads $r\ge \alpha+1$ and, as $\alpha-r/2\in\Z^+$, necessarily $r\ge 4$ and  we can appeal to Lemma~\ref{notc1}. 
\end{proof}

\

Another direction is to consider the situation in which the modular form is not a cusp form. In \cite{jaffard}, S.~Jaffard considered the so-called Riemann's example (see \cite{hardy}, \cite{gerver}, \cite{duistermaat}),
\[
 R(x)=\sum_{n=1}^\infty\frac{\sin(\pi n^2x)}{n^2},
\]
that corresponds to $f_1^s$ when $f$ is the Jacobi theta function, and he determined $\beta(x)$ for $x\in\R$ showing that $R$ is multifractal.

The problem for generic modular forms (not necessarily cusp forms) has been addressed by the second author in \cite{petrykiewicz1} and by the first and third authors in a forthcoming paper (see \cite{ruizcabello}).
Here we consider the behavior at rational values $p/q$. They are cusps for the underlying (congruence) group. 

There are two possible behaviors of a modular form $f$ of weight $r$ when taking vertical limits towards a cusp $x_0=p/q$. 
Either $f(p/q+iy)=o(y^{-r})$ as $y\to 0^+$ or $y^rf(p/q+iy)\to C\ne 0$ as $y\to 0^+$. In the first case we say that $f$ is \emph{cuspidal} at $p/q$ and, accordingly, we say that it is \emph{not cuspidal} at $p/q$ in the second case. The proof of this alternative depends on the Fourier expansion at the cusp \cite{iwaniec}, \cite{chamizo}. In the first case we have indeed an exponential decay $f(p/q+iy)=O(e^{-K/y})$ for some $K>0$.

\

\begin{theorem}\label{pr:notcuspidal}
Let $f(z)=\sum_{n=0}^{\infty}a_ne(nz)$ be a modular form of weight~$r$. Consider $f_\alpha$ with $\alpha>r>0$ and $1+\lfloor\alpha-r\rfloor\le \alpha$. If $f$ is not cuspidal at the irreducible fraction $p/q$, then
$\beta(p/q) = \alpha-r$.
\end{theorem}

\begin{proof}
If $\beta(p/q) > \alpha-r$, then $f_\alpha\in C^{\alpha-r+\delta}(p/q)$ for some $\delta>0$, that we assume small enough to have $\lfloor\alpha-r+\delta\rfloor\le \alpha-1$. By Proposition~\ref{pr:jaffard}
\[
f(x+iy)=
O\big( y^{\delta-r}\big(1+y^{-1}|x-p/q|\big)^{\alpha-r+\delta}\big)
\qquad\text{when }(x,y)\to (p/q,0^+).
\]
Choosing $x=p/q$ we have $y^rf(p/q+iy)=O\big(y^\delta\big)=o(1)$ and it contradicts that $f$ is not cuspidal at $p/q$, therefore we have $\beta(p/q) \leq \alpha-r$.

For the optimality, we proceed as in the proof of Theorem~\ref{pr:cusp_rat}. If we assume that $\beta(p/q)<\alpha-r$, then $f_\alpha\notin C^{\alpha-r-\delta}$ for some $0<\delta<\alpha-r$. By Proposition~\ref{pr:jaffard}, we conclude that 
$f(x+iy) \neq O(y^{-r-\delta}(1+y^{-1}|x-p/q|\big)^{\alpha-r-\delta}$ as $(x,y)\to(p/q,0^+)$. In particular, there exist
$x_n\to p/q$ and $y_n\downarrow 0$ such that
\begin{equation}\label{eq:lowerbound}
\lim_{n\to\infty} |f(x_n+iy_n)|y_n^{r+\delta}(1+y_n^{-1}|x_n-p/q|\big)^{-\alpha+r+\delta}\to \infty.
\end{equation} 
Then we can extract a subsequence such that either $y_n^{-1}|x_n-p/q|$ is bounded or it tends to $\infty$. If $y_n^{-1}|x_n-p/q|$ is bounded, then $y_n^{-1}|x_n-p/q|^2\to 0$, and since $f$ is not cuspidal at $p/q$ \cite[Lemma~3.3]{chamizo} we have $|f(x_n+iy_n)|=O(y_n^{-r})$ which contradicts \eqref{eq:lowerbound}. 
On the other hand if $y_n^{-1}|x_n-p/q|\to \infty$, then \eqref{eq:lowerbound} implies that $|f(x_n+iy_n)|y_n^{r+\delta}\to \infty$ contradicting that $|f(x_n+iy_n)|=O(y_n^{-r})$ as $y_n\to 0^+$.
\end{proof}

\

To finish this section, we are going to prove Lemma~\ref{notc1}. In order to do so, we need an application of basic harmonic analysis (essentially Fej\'er's theorem) and a kind of approximate functional equation for fractional derivatives.

\begin{lemma}\label{fejer_c1}
If $f$ is a cusp form of weight $r>0$,  then 
$f_{r/2+1}\not\in C^{1,0}(\R)$.
\end{lemma}

\begin{lemma}\label{local_frac}
Let $f$ be a cusp form of weight $r\ge 4$, $\gamma$ an element of the underlying congruence group and $I\subset\R$ an open bounded interval not containing the possible pole of $\gamma$.
Then, there exist a nonzero constant $C$ such that 
$f_{r/2+1}(x)-C\big(j_\gamma(x)\big)^2 f_{r/2+1}\big(\gamma(x)\big)\in C^{1,0}(I)$.
\end{lemma}

\begin{proof}[Proof of Lemma~\ref{fejer_c1}]
Assume that $f'_{r/2+1}$ is continuous. By partial integration, its Fourier coefficients are $2\pi ia_n/n^{r/2}$, where $a_n$ are the Fourier coefficients of $f$, and by 
Fej\'er's theorem
\[
 2\pi i
\sum_{n=1}^N
\Big(1-\frac{n}{N}\Big)
\frac{a_n}{n^{r/2}}e(nx)
\underset{\text{unif.}}{\longrightarrow}
f_{r/2+1}'
\qquad
\text{as }N\to\infty.
\]
Taking $L^2[0,1]$ norms, applying Parseval's identity and dropping the terms with $n>N/2$, we have
\[
\frac{1}{\pi^2}\|f_{r/2+1}'\|^2_{L^2[0,1]}
\ge
\sum_{n\le N/2}
\frac{|a_n|^2}{n^r}
\ge
\sum_{n\le N/2-1}
\big(
\frac{1}{n^r}
-\frac{1}{(n+1)^r}
\big)
\sum_{k=1}^n|a_k|^2.
\]
By Lemma~3.2~c) of \cite{chamizo}, the innermost sum is greater than $Cn^r$ for large~$n$. Then  the 
right hand side is of the same order as the harmonic series that diverges. Hence $f_{r/2+1}'$ cannot be continuous on~$[0,1]$. 
\end{proof}

\begin{proof}[Proof of Lemma~\ref{local_frac}]
For the sake of simplicity, along the proof we write $C$ to denote a nonzero constant depending (at most) on $r$ and $\gamma$, not necessarily the same constant each time. 

We assume $\gamma(\infty)\ne\infty$ because otherwise $\gamma$ is an integral translation and the result is trivial ($f$ is $1$-periodic). Then $j_{\gamma}$ and $j_{\gamma^{-1}}$, 
the denominators of the transformations $\gamma$ and $\gamma^{-1}$,
are nonconstant linear functions.

By direct integration of the Fourier series
\begin{equation}\label{eq:fgamma}
f_{r/2+1}(x)=
C\int_0^\infty t^{r/2}f(x+it)\; dt
=
\int_{(x)} (z-x)^{r/2}f(z)\; dz
\end{equation}
where $(\sigma)$ means the ray $\Re(z)=\sigma$,  $\Im(z)>0$.
Note the similarity between this integral and the wavelet transform \eqref{wa_tr}.

By \eqref{eq:modrel} and a change of variables (see \cite[\S2.1]{iwaniec} for the properties of~$j_\gamma$)
\begin{align*}
f_{r/2+1}(x)
&=
C\int_{(x)} (z-x)^{r/2}\big(j_\gamma(z)\big)^{-r}f(\gamma z)\; dz
\\ 
&=
C\int_{S} \big(\gamma^{-1}w-x\big)^{r/2}\big(j_{\gamma^{-1}}(w)\big)^{r-2}f(w)\; dw
\end{align*}
where $S$ is the semicircle in the upper half plane connecting $\gamma(x)$ and $\gamma(\infty)$. By the residue theorem and the vanishing at $i\infty$, the integral can be replaced by two integrals along the rays $\big(\gamma(x)\big)$ and $\big(\gamma(\infty)\big)$. As $\gamma(\infty)$ is a cusp, $f(w)$ has exponential decay at $\gamma(\infty)$ and the latter integral defines a regular function $h=h(x)$. Hence
\[
f_{r/2+1}(x)
=
C\int_{(\gamma(x))} \big(\gamma^{-1}w-x\big)^{r/2}\big(j_{\gamma^{-1}}(w)\big)^{r-2}f(w)\; dw
+h(x).
\]
Substituting the relation  
$\big(\gamma^{-1}w-x\big)j_{\gamma^{-1}}(w)
=\big(w-\gamma(x)\big)j_{\gamma}(x)$ \cite[(2.4)]{iwaniec} in this formula, it reads
\[
f_{r/2+1}(x)
=
C
\big(j_{\gamma}(x)\big)^{r/2}
\int_{(\gamma(x))} 
\big(w-\gamma(x)\big)^{r/2}\big(j_{\gamma^{-1}}(w)\big)^{r/2-2}f(w)\; dw
+h(x).
\]
If $r=4$, the exponent $r/2-2$ vanishes and we are done using \eqref{eq:fgamma}. Then we assume $r>4$. 

Writing  $w=\gamma(x)+it$, the integral is, up to a constant,
\[
 \int_0^\infty
\big(j_{\gamma^{-1}}(\gamma(x)+it)\big)^{r/2-2}
t^{r/2}f\big(\gamma(x)+it\big)
\; dt
=
\int_0^\infty u\; dv
\]
where
\[
 u=
\big(j_{\gamma^{-1}}(\gamma(x)+it)\big)^{r/2-2}
\qquad\text{and}\qquad
v=-\int_t^{\infty}y^{r/2}f\big(\gamma(x)+iy\big)\; dy.
\]
Integrating by parts, recalling \eqref{eq:fgamma} and using once again the properties of~$j_\gamma$ to simplify, we obtain
\begin{equation}\label{eq:forbit}
f_{r/2+1}(x)
=
C
\big(j_{\gamma}(x)\big)^{2}
f_{r/2+1}\big(\gamma(x)\big)+h(x)
-
\big(j_{\gamma}(x)\big)^{r/2}
\int_0^\infty
v\; du.
\end{equation}
It remains to prove that the integral defines a $C^{1,0}(I)$ function of $x$.

Changing the order of integration, we have 
\begin{align*}
\int_0^\infty
v\; du
&=
C\int_0^\infty\int_0^y 
\big(j_{\gamma^{-1}}(\gamma(x)+it)\big)^{r/2-3}
y^{r/2}f(\gamma(x)+it)\; dtdy
\\ 
&=
C\int_{(\gamma(x))}
g(z,x)f(z)\; dz 
\end{align*}
where
\[
 g(z,x)=
\Big(
\big(j_{\gamma^{-1}}(\gamma(x)+iy)\big)^{r/2-2}
-
\big(j_{\gamma^{-1}}(\gamma(x))\big)^{r/2-2}
\Big)
\big(z-\gamma(x)\big)^{r/2}.
\]
As $g\big(\gamma(x),x\big)=0$,
\[
 \frac{d}{dx}
\int_0^\infty
v\; du
=
C\int_{(\gamma(x))}
\frac{\partial g}{\partial x}
(z,x)f(z)\; dz.
\]
This partial derivative is a continuous function and it is $O\big(t^{r/2}\big)$ as $t\to 0^+$ when evaluated at $z=\gamma(x)+it$. 
Note that
$j_{\gamma^{-1}}\big(\gamma(x)\big)$ is a regular linear fractional transformation because the pole of $\gamma$ is not in~$I$.
On the other hand, we know \cite[(5.3)]{iwaniec} that $f(\gamma(x)+it)=O\big(t^{-r/2}\big)$ as $t\to 0^+$ and it has an exponential decay for $t\to \infty$. 
Hence the integral in \eqref{eq:forbit} is in $C^{1,0}(I)$ and the result is proved.
\end{proof}

\begin{proof}[Proof of Lemma~\ref{notc1}]
Assume that $f_{r/2+1}\in C^{1,0}(I_0)$ for certain open interval $I_0$. Take an arbitrary real number~$x$. We can find $\gamma$ such that $\gamma(x)\in I_0$ (recall that the underlying group is a congruence group). Hence $f_{r/2+1}\circ \gamma\in C^{1,0}(I)$ for some open neighborhood $I$ of $x$. By  Lemma~\ref{local_frac}, the same applies for  $f_{r/2+1}$. As $x$ is arbitrary, we get a contradiction with  Lemma~\ref{fejer_c1}.
\end{proof}

\section{Examples and computer plots}

The Shimura-Taniyama-Weil conjecture allows to assign a cusp form to each elliptic curve over $\Q$. For instance,
\[
E:\; y^2+yx+y=x^3+4x-6
\]
gives a cusp form $f(z)=\sum_{n=1}^\infty a_n e(nz)$ of weight~2 in $\Gamma_0(14)$ where $a_p= p+1-N_p$ with $N_p$ the number of points of $E$ over $\mathbb{F}_p$ for a prime $p>7$ and the rest of the $a_n$ are determined by multiplicative properties \cite{knapp}. 

The series
\[
F(x)=\sum_{n=1}^\infty \frac{a_n}{n^\delta}\sin(2\pi nx)
\]
converges absolutely for any $\delta>1$ \cite[Lemma~3.2]{chamizo} and, according to Theorem~\ref{pr:modforms}, defines a continuous function such that $\beta(x_0)=\beta^{*}(x_0)=\beta^{**}(x_0)=\delta -1$ for $x_0\in\R\setminus\Q$. For instance, for $\delta =7/4$ the H\"older exponents are 3/4 and we have a fractal-like graph. Its global aspect for $0\le x\le 1/2$ is
\begin{center}
\begin{tabular}{c}
\includegraphics[width=90mm,height=60mm]{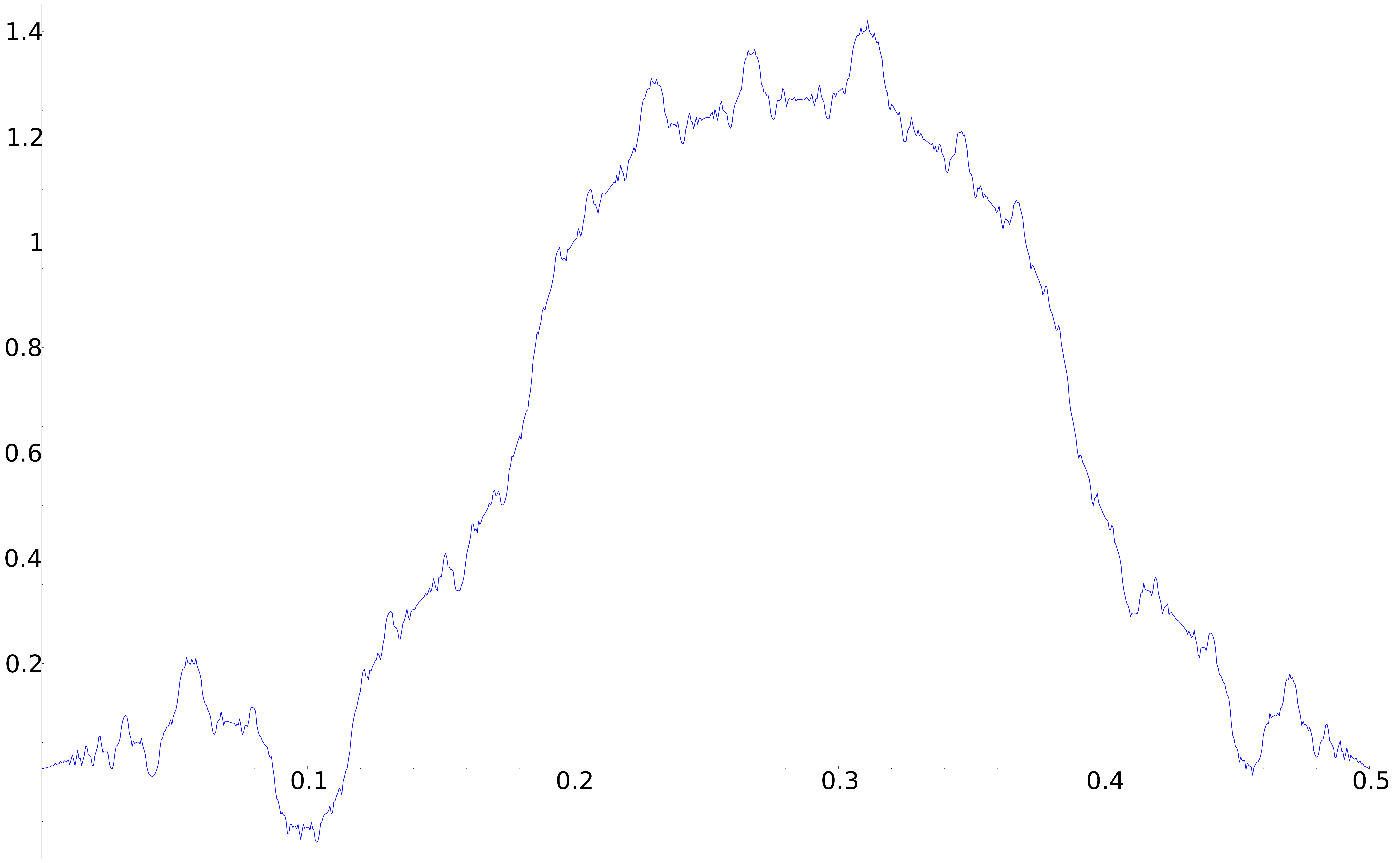}
\\
{\sc Fig.\;1}\quad
\sf Graph of $\sum_{n=1}^\infty a_n n^{-\delta}\sin(2\pi nx)$ for $0\le x\le 1/2$
\end{tabular}
\end{center}

\medskip

It is not a multifractal by Corollary~\ref{cr:modforms} (but it is fractal for $\delta\not\in\Z$). It can be proved \cite[Theorem~2.2]{chamizo} that it is differentiable in $\Q$. Consequently we see a different behavior when we zoom the graph around $\sqrt{2}-1$ and $0$, as the following figures show:
\begin{center}
\begin{tabular}{c}
\includegraphics[width=54mm,height=34mm]{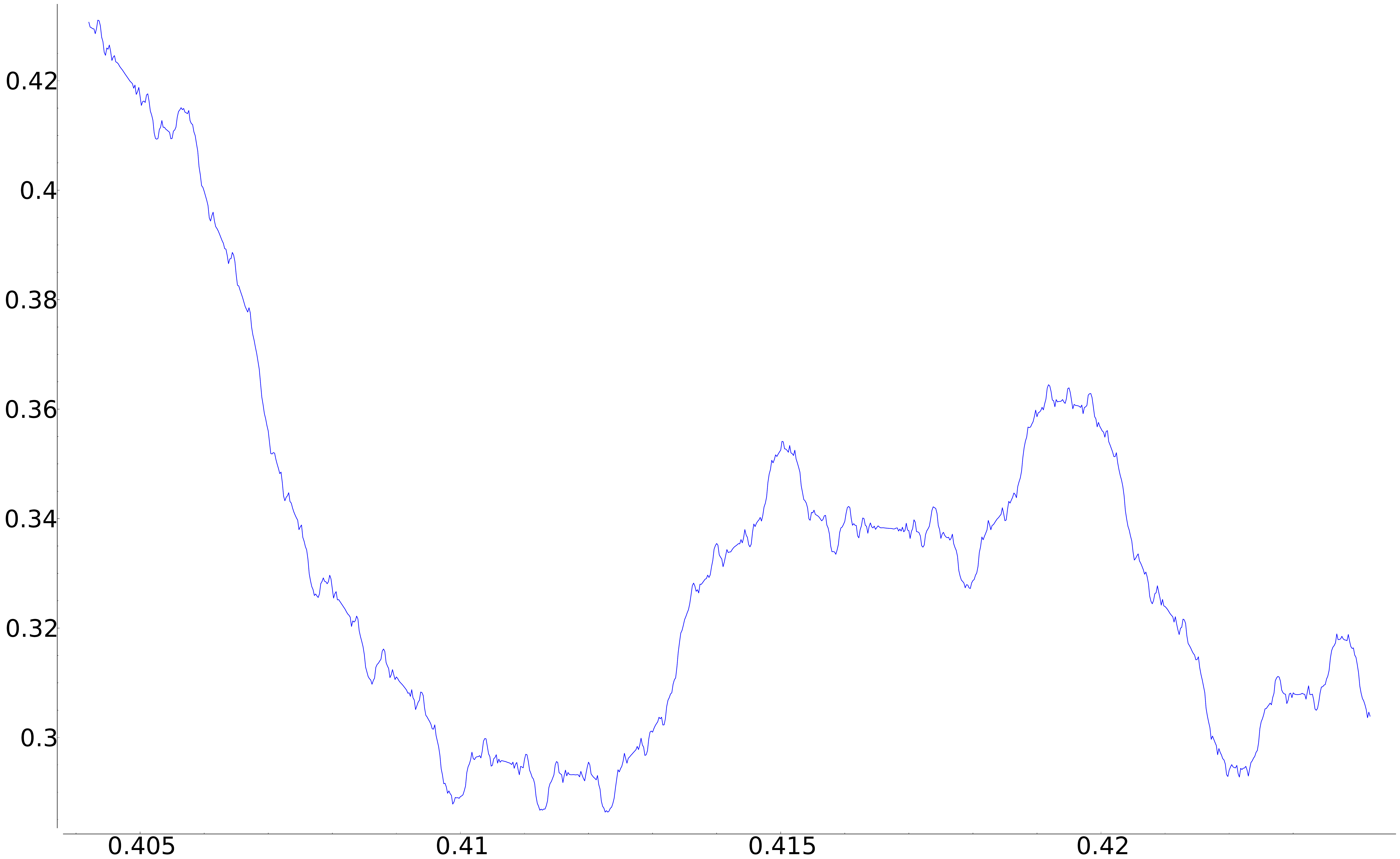}
\\
{\sc Fig.\;2}\quad
\sf Zoom $[0.4042,0.4242]$
\end{tabular}
\begin{tabular}{c}
\includegraphics[width=54mm,height=34mm]{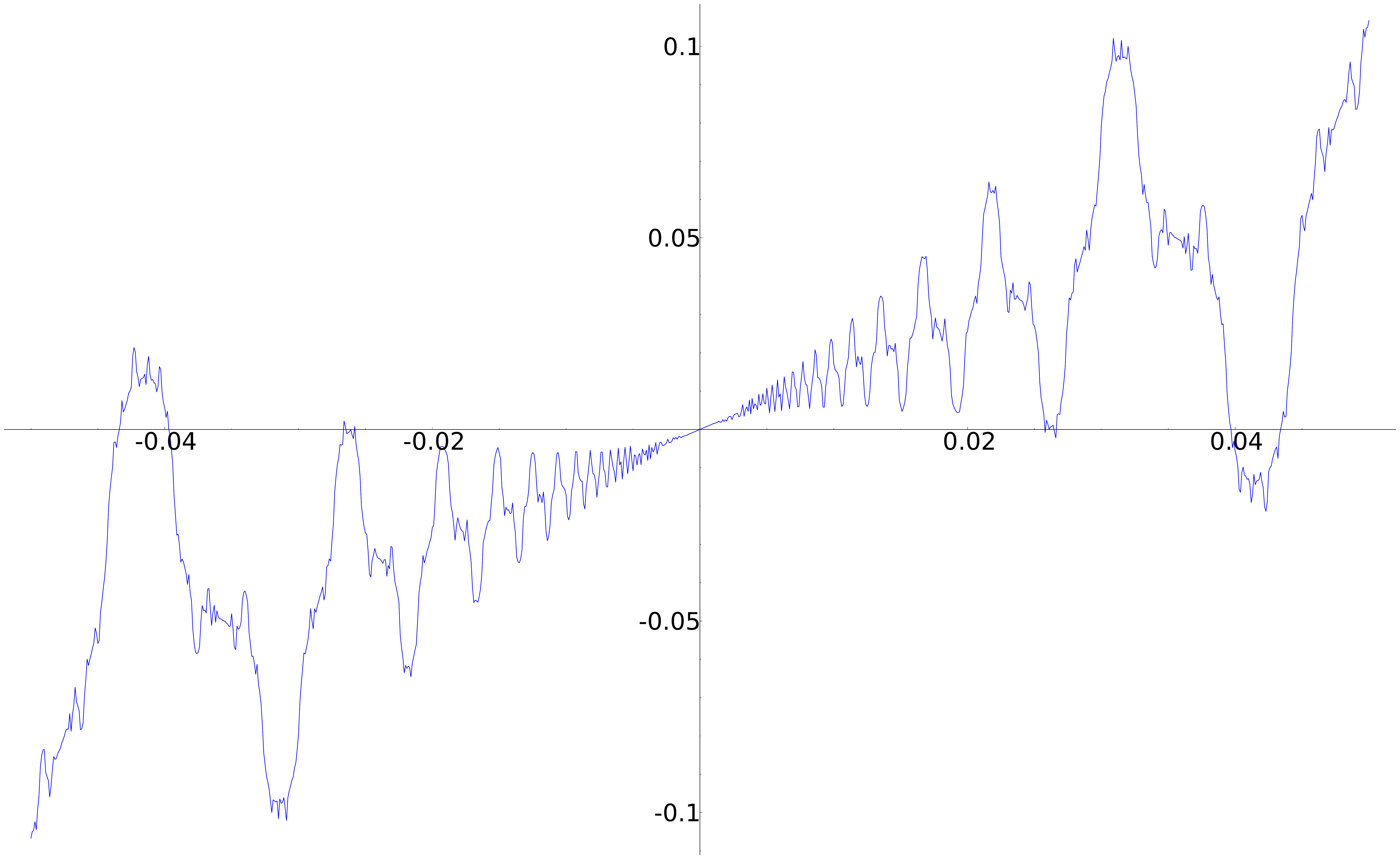}
\\
{\sc Fig.\;3}\quad
\sf Zoom $[-0.05,0.05]$
\end{tabular}
\end{center}

\medskip

The theta series produce quite explicit and simple examples. 
For instance,
\begin{equation}\label{theta_example}
 F(x)= \sum_{n=1}^\infty
 \frac{c_n}{n^{\delta}}\cos(2\pi n^2x)
 \quad\text{with}\quad
 c_n=\begin{cases}
      1&\text{if $n\equiv\pm 1\pmod{12}$}
      \\
      -1&\text{if $n\equiv\pm 5\pmod{12}$}
     \end{cases}
\end{equation}
verifies 
\[
\beta(x_0)=\beta^{*}(x_0)=\beta^{**}(x_0)=\frac{2\delta-1}{4} 
\qquad\text{for }\quad\delta>\frac 12.
\]
This is a consequence of Theorem~\ref{pr:modforms} with $\alpha=\delta/2$ since it is known \cite{SeSt} that $f(z)=\sum c_n e(n^2z)$ is a cusp form of weight $r=1/2$ for the group $\Gamma_0(576)$. Taking $\delta =2$ (figure on the left) we obtain something that resembles the Riemann's example (figure on the right) that also corresponds to a theta function of weight $r=1/2$. 
\begin{center}
\begin{tabular}{c}
\includegraphics[width=54mm,height=34mm]{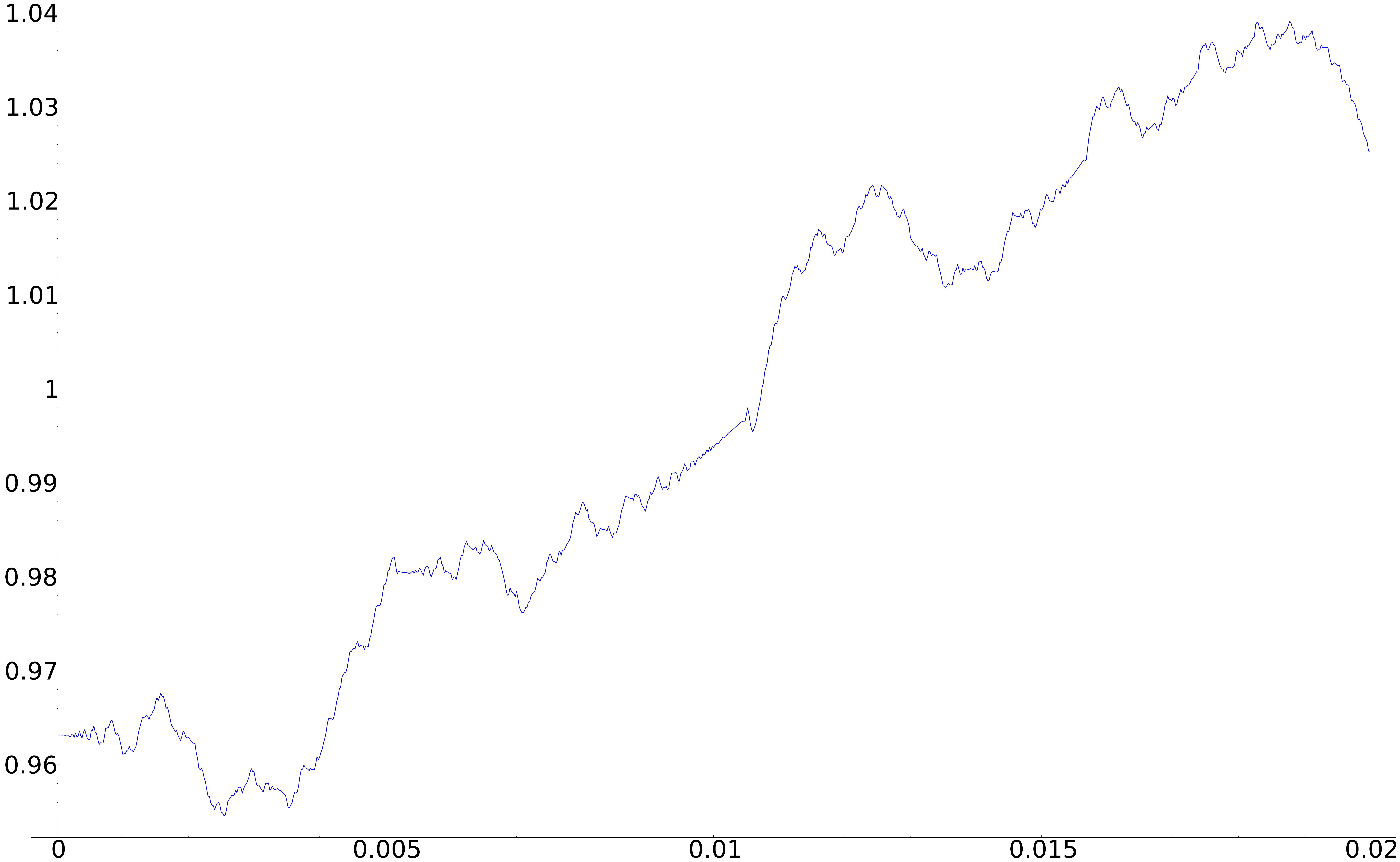}
\\
{\sc Fig.\;4}\quad
\sf Theta cusp form $[0,0.02]$
\end{tabular}
\begin{tabular}{c}
\includegraphics[width=54mm,height=34mm]{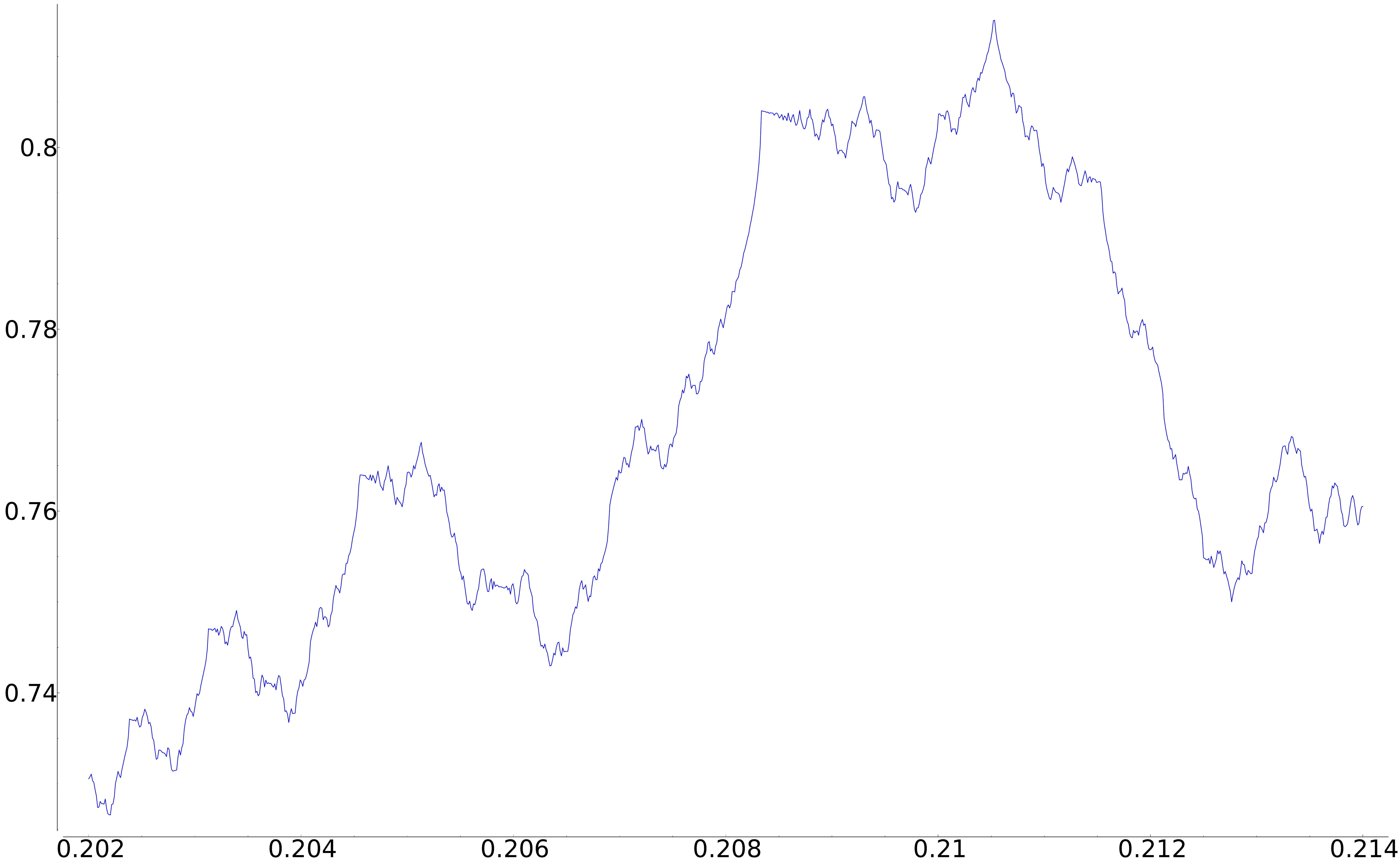}
\\
{\sc Fig.\;5}\quad
\sf Riemann's example (detail) 
\end{tabular}
\end{center}

\medskip

But we know by Corollary~\ref{cr:modforms} that $F$ is not a multifractal and the main result of \cite{jaffard} proves that Riemann's example is a multifractal. 

\

In some cases we can deduce a complete determination of the spectrum of singularities. 

The polynomial $P(x,y)= x^4+y^4-6x^2y^2$ is harmonic ($\Delta P = 0$) and the theory (see \cite[p.14]{sarnak}) assures that
\[
f(z)=
\sum_{n\in\Z}
\sum_{m\in\Z}
P(n,m)e\big((n^2+m^2)z\big)
\]
is a cusp form of weight $\deg P +2\cdot 1/2=5$. 

Consider the function 
\[
F(x)=
\sum_{n^2+m^2\ne 0}
\frac{n^4+m^4-6n^2m^2}{(n^2+m^2)^{13/4}}\cos\big((n^2+m^2)x\big).
\]
Its graph for $|x|\le 1/2$ is
\begin{center}
\begin{tabular}{c}
\includegraphics[width=90mm,height=60mm]{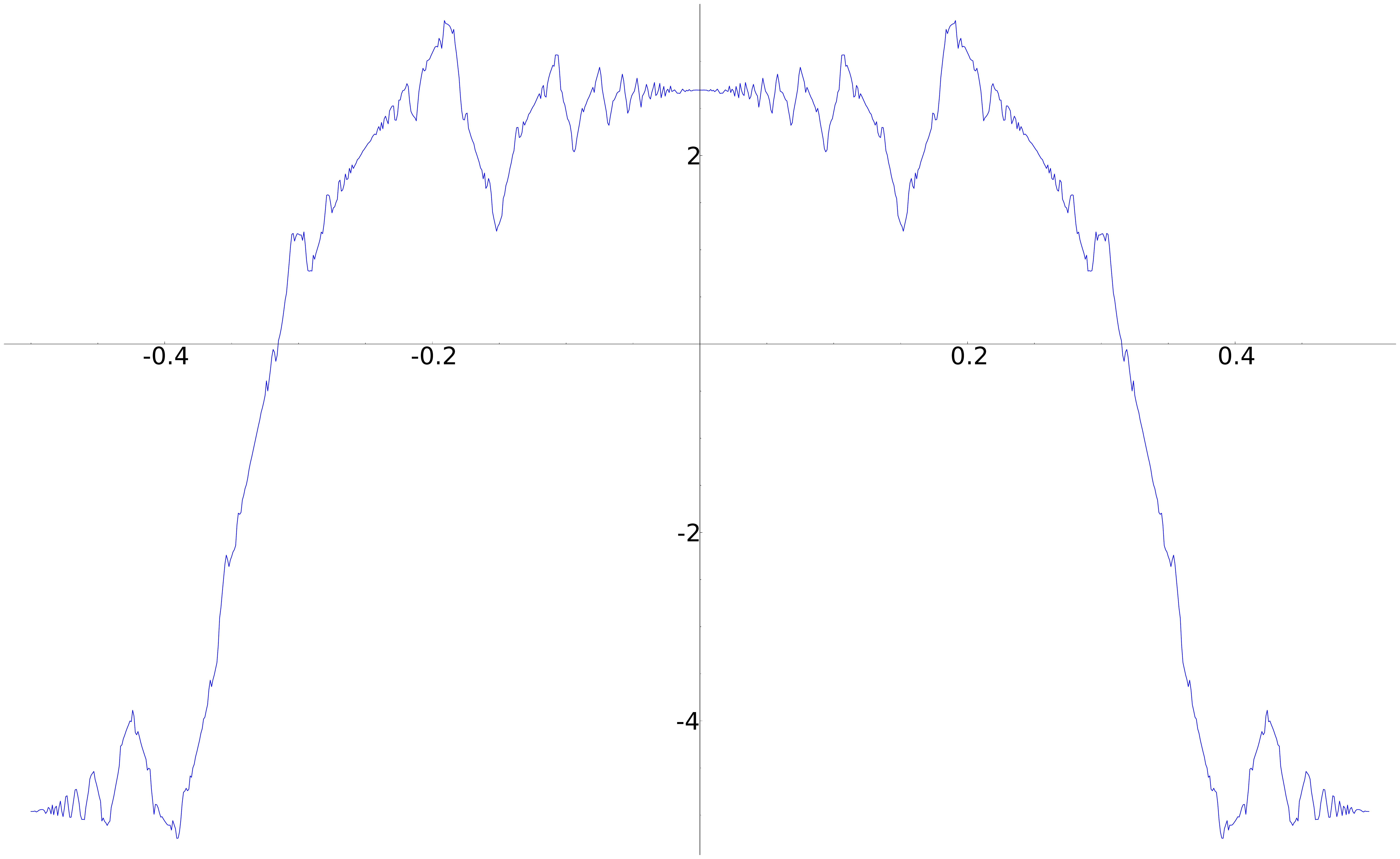}
\\
{\sc Fig.\;6}\quad
\sf Graph of $\sum_{n^2+m^2\ne 0}
\frac{n^4+m^4-6n^2m^2}{(n^2+m^2)^{13/4}}\cos\big((n^2+m^2)x\big)$
\end{tabular}
\end{center}

\medskip

Note that $F(x)=f_{13/4}^c(x)$. Theorem~\ref{pr:cusp_rat} implies $\beta(x_0)=3/2$ for $x_0\in\Q$ and using Theorem~\ref{pr:modforms} for the irrationals, we get
\[
 d_F(\delta)=
 \begin{cases}
  1& \text{if }\delta=3/4
  \\
  0& \text{if }\delta=3/2
  \\
  -\infty& \text{otherwise.}
 \end{cases}
\]

\

We finish with two examples that do not come from cusp forms.

The complex analog of Riemann's example $F(x)=\sum_{n=1}^\infty n^{-2}e(n^2x)$ corresponds, up to $1/2$ factor, to 
$f(z)=\sum_{n\in\Z} e(n^2z)$ and $\alpha=1$. It is well-known \cite{iwaniec} that $f$ is an automorphic form of weight $r=1/2$. The underlying group is $\Gamma_0(4)$ and the cusps are always equivalent to $0$, $1/4$ or $1/2$ under the action of this group, being $f$ cuspidal at $1/2$ and not cuspidal at $0$ and $1/4$. The orbits of these latter numbers run on the irreducible fractions $p/q$ such $q$ is odd or $4|q$. Theorem~\ref{pr:notcuspidal}  assures that $f\in C^{1/2}(x_0)$ at those points (see \cite{duistermaat} for some graphs).

A simpler example related to \cite{petrykiewicz1} is $F(x)=\sum_{n=1}^\infty n^{-\alpha}\sigma_{k-1}(n)e(nx)$ with $k$ even, $\alpha>k>2$ and $\sigma_{k-1}(n)=\sum_{d\mid n} d^{k-1}$. It is associated to the classical Eisenstein series of weight $k$ for the full modular group \cite{iwaniec}.
Clearly it is not cuspidal at~$0$ and as all the rational numbers are in the same orbit, by Theorem~\ref{pr:notcuspidal} we have $F\in C^{\alpha-k}(x_0)$ for every $x_0\in\Q$. The same result holds for $\Re(F)$ and $\Im(F)$.

\bibliography{bibarts}{}
\bibliographystyle{alpha}

\end{document}